\documentclass[11pt, a4paper]{amsart}

\usepackage{amsmath}
\usepackage{amssymb}
\usepackage{graphicx}
\usepackage{color}
\usepackage{url}
\usepackage{hyperref}
\hypersetup{
    colorlinks,
    citecolor=black,
    filecolor=black,
    linkcolor=black,
    urlcolor=black
}

\newtheorem{thm}{Theorem}
\newtheorem{prop}[thm]{Proposition}
\newtheorem{lem}[thm]{Lemma}
\newtheorem{remark}[thm]{Remark}

\begin{document}
\title{On the genus defect of positive braid knots}
\author{Livio Liechti}
\address{Department of Mathematics, University of Fribourg, Chemin du Mus\'ee, 1700 Fribourg, Switzerland}
\email{livio.liechti@unifr.ch}

\begin{abstract} 
We show that the difference between the Seifert genus and the topological~$4$-genus of a prime positive braid knot is bounded from below 
by an affine function of the minimal number of strands among positive braid representatives of the knot.
We deduce that among prime positive braid knots, the property of having such a genus difference less than any fixed constant is characterised by finitely many forbidden surface minors.  
\end{abstract}
\maketitle

\section{Introduction}

The discrepancy between the smooth and the topological category in dimension four distinctly manifests itself in the behaviour of the smooth and the topological~$4$-genus of positive braid knots. 

Let~$g_4^\mathrm{top}$,~$g_4^\mathrm{smooth}$ and~$g$ be the topological~$4$-genus, the smooth~$4$-genus and the Seifert genus, respectively. Then, we have 
$$\frac{|\sigma(K)|}{2}\le g_4^\mathrm{top}(K)\le g_4^\mathrm{smooth}(K) \le g(K)$$
for any knot~$K$, where~$\sigma$ is the signature invariant. The first inequality is due to Kauffman and Taylor~\cite{KT}, and the others follow quickly from the definitions, which we will give shortly.
We say the knot~$K$ has \emph{maximal} signature, topological~$4$-genus or smooth~$4$-genus, if the above inequality between the respective invariant and the Seifert genus~$g$ is an equality.  

Positive braid knots have maximal smooth~$4$-genus by the resolution of the Thom conjecture due to Kronheimer and Mrowka~\cite{KM}, and Rudolph's extension to strongly quasipositive knots~\cite{Ru2}. 
In strong contrast, a positive braid knot that fails to have maximal signature also fails to have maximal topological~$4$-genus by a result of the author~\cite{L2}. 
Using Baader's classification of positive braid knots with maximal signature~\cite{Ba}, 
this implies that a positive braid knot of positive braid index greater than or equal to four never has maximal topological~$4$-genus.

The aim of this article is to study the quantity~$g-g_4^\mathrm{top}$, which we call the \emph{genus defect}, 
in the context of positive braid knots, where it also equals~$g_4^\mathrm{smooth}-g_4^\mathrm{top}$.
We show that the genus defect of a positive braid knot is bounded from below by an affine function of the positive braid index (Theorem~\ref{linearity_thm}),
and prove the existence of a characterisation by finitely many forbidden surface minors for the property to have genus defect smaller than or equal to~$c$, where~$c$ is any fixed constant (Theorem~\ref{forbiddenminorchar}).

\subsection{The genus defect}
Let~$L$ be an oriented link in the~$3$-sphere. The \emph{Seifert genus}~$g(L)$ of the link~$L$ is the minimal genus among connected compact oriented surfaces in the~$3$-sphere having~$L$ as boundary. 
The \emph{topological~$4$-genus}~$g_4^\mathrm{top}(L)$ of the link~$L$ 
is the minimal genus among properly, topologically locally-flatly embedded connected compact oriented surfaces in the~$4$-ball having the link~$L$ in the~$3$-sphere as boundary. 
The \emph{smooth~$4$-genus}~$g_4^\mathrm{smooth}(L)$ of the link~$L$ is defined analogously, by replacing ``topologically locally-flatly" with ``smoothly". 

For large enough parameters~$p$ and~$q$, the genus defect of the torus knot~$T(p,q)$ is greater than one quarter of its Seifert genus by a result of Baader, Feller, Lewark and the author~\cite{BFLL}.
Our first result applies to a more general class of knots, but draws a weaker conclusion. It states that the genus defect of any positive braid knot
is bounded from below by an affine function of the positive braid index.
Here, a positive braid knot on~$n+1$ strands is defined to be the closure of a braid given by a positive word in the braid generators~$\sigma_1,\ \dots,\sigma_n$ (see Section~\ref{background} for a precise definition),
and the \emph{positive braid index}~$b$ is the minimal number of strands among positive braid representatives of the knot.

\begin{thm}
\label{linearity_thm}
For a prime positive braid knot~$K$ of positive braid index~$b$, we have~$$g(K) - g_4^\mathrm{top}(K) \ge \left \lfloor{\frac{b}{16}}\right \rfloor .$$
\end{thm}

The proof of Theorem~\ref{linearity_thm} critically uses the fact that genus defect is inherited from surface minors (defined in Section~\ref{surfaceminorsec} below).
More precisely, we study the linking graph of positive braid knots, a concept implicitly used by Baader, Feller, Lewark and the author~\cite{Ba, BFLL, L2} and rigorously defined by Baader, Lewark and the author~\cite{BLL}. We deduce a series of lemmas to find an affinely increasing (in the positive braid index) number of certain surface minors~$\widetilde T$,~$\widetilde E$ or~$\widetilde X$ 
of the canonical Seifert surface of any prime positive braid knot. Finally, for the boundary links of the surfaces~$\widetilde T$,~$\widetilde E$ and~$\widetilde X$, 
a positive genus defect was established by the author~\cite{L2} using Freedman's disc theorem~\cite{Free}.

We note that it is essential to assume that~$K$ is a knot in Theorem~\ref{linearity_thm}. Indeed,  
for prime positive braid links, there exists no nontrivial lower bound for the genus defect~$g-g_4^\mathrm{top}$ in terms of the minimal positive braid index:
in Remark~\ref{nolowerbound}, we give examples of prime positive braid links of arbitrary positive braid index for which~$g=g_4^\mathrm{top}$. 

\subsection{Surface minors}
\label{surfaceminorsec}
The \emph{surface minor relation} is a partial order on embedded surfaces in the~$3$-sphere, where a surface~$\Sigma_1$ is a \emph{minor} of another surface~$\Sigma_2$ 
if~$\Sigma_1$ can be isotoped in the~$3$-sphere to an incompressible subsurface of~$\Sigma_2$. It was introduced by Baader and Dehornoy in the context of {Seifert surfaces} of links~\cite{BaDe}.

The surface minor relation is well-suited for the study of properties of links which are inherited from Seifert surface minors,
for example having genus defect:
by a surgery argument,~$g - g_4^\mathrm{top}\ge c$ is inherited from surface minors if we restrict ourselves to Seifert surfaces which realise the genus of the links, see, for example,~\cite{BFLL}. 
Our second result establishes the existence of a forbidden minor characterisation for the genus defect of prime positive braid knots.
We use that prime positive braid knots have a canonical genus-minimising Seifert surface (described in Section~\ref{background}) by a result of Stallings~\cite{St}.
 
\begin{thm}
\label{forbiddenminorchar}
Among prime positive braid knots, for any~$c\ge0$, having a genus defect~$g-g_4^\mathrm{top}\le c$ is characterised by finitely many forbidden surface minors of the canonical Seifert surface.
\end{thm}

Having maximal topological~$4$-genus has previously been characterised for prime positive braid knots by the author~\cite{L2}. 
More precisely, this has been done by giving four explicit forbidden surface minors~$\widetilde T$,~$\widetilde E$,~$\widetilde X$ and~$\widetilde Y$ of the canonical Seifert surface. 
From this perspective, Theorem~\ref{forbiddenminorchar} is a non-explicit generalisation of this result. 

In order to prove Theorem~\ref{forbiddenminorchar}, we use the theory of well-quasi-orders.
Higman's Lemma states that finite words in a finite alphabet with the subword partial order are \emph{well-quasi-ordered}, that is, there exists no infinite antichain and no infinite descending chain~\cite{Higman}.
In this context, a \emph{subword} of a word~$w$ is obtained by deleting any number of letters anywhere in~$w$. 
For well-quasi-ordered sets, properties that are passed on to \emph{minors}, that is, smaller elements with respect to the partial order, are of special interest: 
they can be characterised by finitely many forbidden minors. 
Indeed, if infinitely many forbidden minors were necessary to characterise such a property, then they would constitute an infinite antichain.
Baader and Dehornoy noted that restricting to the positive braid monoid on a certain number of strands, 
Higman's Lemma states that the subword partial order is a well-quasi-order, and
it directly follows that their canonical Seifert surfaces are well-quasi-ordered by the surface minor relation~\cite{BaDe}.
However, the subword partial order on the positive braid monoid is not a well-quasi-order if we do not restrict to a fixed number of strands: 
for example, already~$\sigma_1,\sigma_2,\sigma_3 \dots$ is an infinite antichain.
The key input for the proof of Theorem~\ref{forbiddenminorchar} is a reduction to the case of restricted braid index, so we can apply Higman's Lemma. 
Such a reduction can be achieved with the help of Theorem~\ref{linearity_thm}.

\begin{remark}\emph{
Theorem~\ref{linearity_thm} and Theorem~\ref{forbiddenminorchar} answer two questions asked by the author~\cite{L2}. 
In the context of positive braids and the surface minor relation, there is another relevant question, asked by Baader and Dehornoy~\cite{BaDe}: 
are canonical Seifert surfaces of positive braids with the surface minor relation well-quasi-ordered?
While it does not answer the question of Baader and Dehornoy, 
Theorem~\ref{forbiddenminorchar} directly gives the application a positive answer would yield for the genus defect~$g-g_4^\mathrm{top}$ 
of positive braid knots.
}\end{remark}

\begin{remark}\emph{
Our proof of Theorem~\ref{linearity_thm} is slightly stronger as it in fact gives the stated bound for the algebraic genus~$g_\mathrm{alg}$, defined by Feller and Lewark~\cite{FL},
which in turn is an upper bound for the topological~$4$-genus by Freedman's disc theorem~\cite{Free}.
Furthermore, Theorem~\ref{linearity_thm} implies the same bound also for the signature defect~$g-|\sigma|/2$ of prime positive braid knots,
by the bound due to Kauffman and Taylor~\cite{KT}.
Theorem~\ref{forbiddenminorchar} consequently holds for the algebraic genus defect and the signature defect as well.
}\end{remark}

\noindent
\textbf{Organisation}
In Section~\ref{background}, we introduce the necessary background on positive braids, their canonical Seifert surfaces and linking graphs, and the surfaces~$\widetilde T$,~$\widetilde E$ and~$\widetilde X$. 
Section~\ref{path_section} and Section~\ref{minor_section} are devoted to finding surface minors~$\widetilde T$,~$\widetilde E$ or~$\widetilde X$ of the canonical Seifert surfaces of positive braid knots
by considering induced subgraphs of the linking graph.
In Section~\ref{linearityproof_section} and Section~\ref{minortheory_section}, we finally prove Theorem~\ref{linearity_thm} and Theorem~\ref{forbiddenminorchar}, respectively.\\

\subsection*{Acknowledgements.}
I warmly thank Sebastian Baader, Peter Feller and Lukas Lewark for many inspiring discussions on the subject of this article.
I also thank an anonymous referee for their helpful comments, in particular for a comment that has led to a stronger form of Remark~\ref{nolowerbound}.
The author was supported by the Swiss National Science Foundation (grant no.\ 175260).

\section{Positive braids and the linking pattern}
\label{background}

A \emph{positive braid on~$n+1$ strands} is given by a \emph{positive braid word} in~$n$ generators, that is, a word in positive powers of the generators~$\sigma_1,\dots,\sigma_n$.
Usually, a positive braid is defined to be such a word up to braid relations~$\sigma_i\sigma_j=\sigma_j\sigma_i$ for~$\vert i-j\vert\ne1$ and~$\sigma_i\sigma_{i+ 1}\sigma_i=\sigma_{i+ 1}\sigma_i\sigma_{i+ 1}$
for~$1\le i<n$.
As this difference is not crucial for our purposes, we often blur the distinction between positive braids and words representing them.
We usually say ``positive braid" when in fact it would be precise to say ``positive braid word". 
A positive braid~$\beta$ can be represented geometrically by taking~$n+1$ strands and inserting a positive crossing between the~$i$th and~$i+1$st strand for every occurrence of~$\sigma_i$.
A~\emph{positive braid link}~$\widehat\beta$ is the closure of the geometric representation of a positive braid~$\beta$. 
See Figure~\ref{Xtildebraid} for an example of the geometric representation and the closure 
of the positive braid given by~$\beta=\sigma_1^2\sigma_2^2\sigma_1\sigma_3\sigma_2^2\sigma_3$.
\begin{figure}[h]
\begin{center}
\def\svgwidth{220pt}
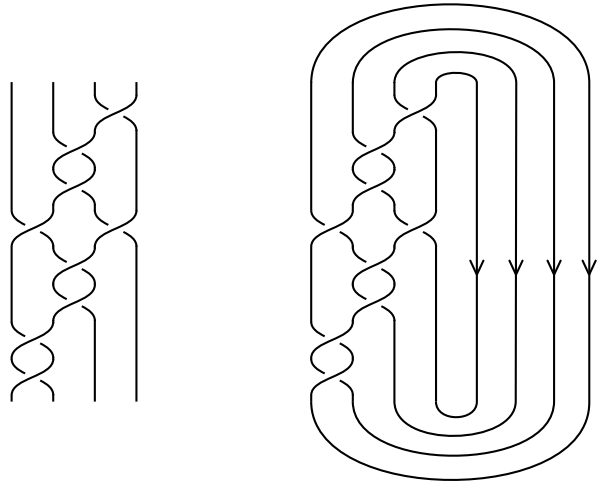
\caption{The positive braid link associated with the word~$\sigma_1^2\sigma_2^2\sigma_1\sigma_3\sigma_2^2\sigma_3$.}
\label{Xtildebraid}
\end{center}
\end{figure}
There is a unique (up to isotopy) genus-minimising Seifert surface~$\Sigma(\beta)=\Sigma(\widehat\beta)$ for each non-split positive braid link~$\widehat\beta$,
by a theorem of Stallings~\cite{St}. We call the surface~$\Sigma(\beta)$ the \emph{canonical Seifert surface} of~$\widehat\beta$. 
It is obtained by taking~$n$ discs and connecting them with a curved handle for every occurrence of~$\sigma_i$. On the left in Figure~\ref{Xtildebrick},
the surface~$\Sigma(\beta)$ is depicted for~$\beta=\sigma_1^2\sigma_2^2\sigma_1\sigma_3\sigma_2^2\sigma_3$.
\begin{figure}[h]
\begin{center}
\def\svgwidth{320pt}
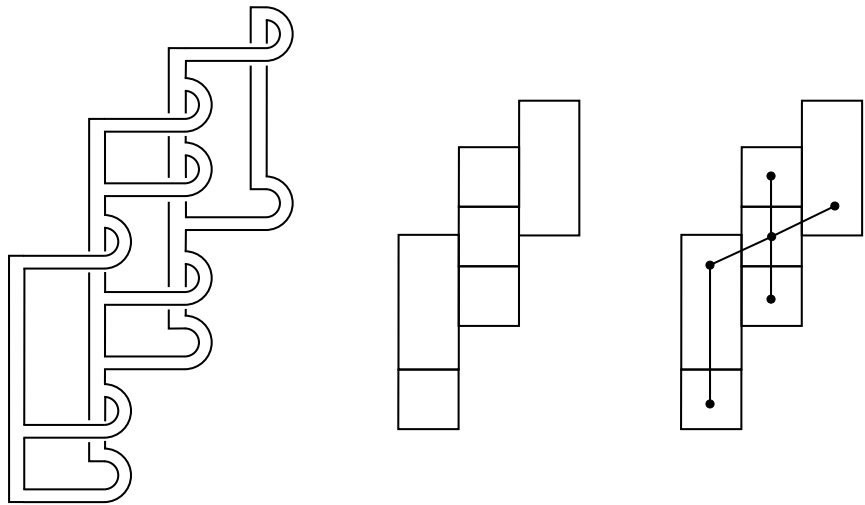
\caption{Retracting the canonical Seifert surface to bricks, and the associated linking pattern.}
\label{Xtildebrick}
\end{center}
\end{figure}
By construction, the canonical Seifert surface retracts to a collection of rectangles in the plane, called the \emph{brick diagram}, as shown in Figure~\ref{Xtildebrick}. 
From this diagram, we construct the \emph{linking pattern}~$\mathcal{P}(\beta)$, a plane graph, by the following rules. There is one vertex for each rectangle. 
Furthermore, two vertices are connected by an edge exactly if the corresponding rectangles share a horizontal side or if two vertical sides overlap partially
(so that the intersection is not equal to one of the sides), see Figure~\ref{Xtildebrick} for an example. 
If there is an edge between two vertices of the linking pattern, we also say that the two corresponding rectangles~\emph{link}.
We note that the linking pattern does depend on the braid word, not only on the positive braid link. 
We also mention that positive braid links are visually prime by a result of Cromwell~\cite{Cro}, 
so a positive braid link is prime exactly if the linking pattern is connected. 

\subsection{The surface minors~$\widetilde T$,~$\widetilde E$ and~$\widetilde X$.}

Let~$\widetilde T$,~$\widetilde E$ and~$\widetilde X$ be the canonical Seifert surfaces
\begin{align*}
\widetilde T &= \Sigma(\sigma_1^5\sigma_2\sigma_1^4\sigma_2),\\
\widetilde E &= \Sigma(\sigma_1^7\sigma_2\sigma_1^3\sigma_2),\\
\widetilde X &= \Sigma(\sigma_1^2\sigma_2^2\sigma_1\sigma_3\sigma_2^2\sigma_3).
\end{align*}
The surface~$\widetilde X$ is depicted in Figure~\ref{Xtildebrick}. Figure~\ref{tildes} in addition shows the brick diagrams and the linking patterns for~$\widetilde T$ and~$\widetilde E$.
\begin{figure}[h]
\begin{center}
\def\svgwidth{170pt}
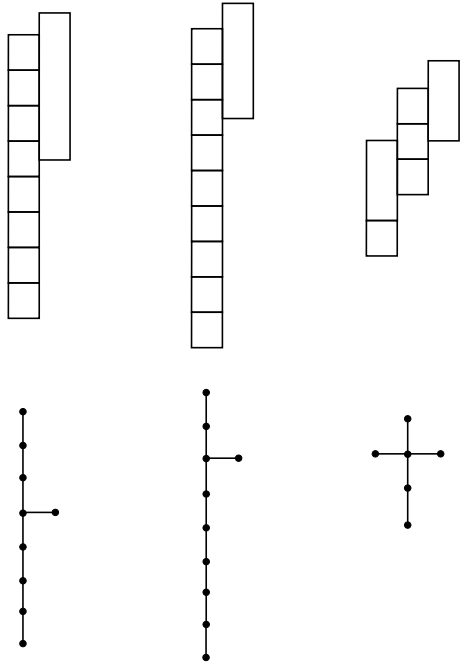
\caption{The brick diagrams and the linking patterns~$\Gamma_{\widetilde T}$,~$\Gamma_{\widetilde E}$ and~$\Gamma_{\widetilde X}$ of~$\widetilde T$,~$\widetilde E$ and~$\widetilde X$, respectively.}
\label{tildes}
\end{center}
\end{figure}
By a result due to the author~\cite{L2}, we have~$g_4^\mathrm{top} = g - 1$
for~$\partial\widetilde T$,~$\partial\widetilde E$ and~$\partial\widetilde X$. 
Since genus defect is inherited from surface minors, we can deduce genus defect of a positive braid knot 
by finding surface minors~$\widetilde T$,~$\widetilde E$ and~$\widetilde X$ in the canonical Seifert surface.
We do this with the help of the following lemma.  
An \emph{induced subgraph} of a graph~$\Gamma$ consists of a subset of the vertices of~$\Gamma$ and all the edges of~$\Gamma$ between those vertices.

\begin{lem}
If the linking pattern~$\mathcal{P}(\beta)$ for a positive braid~$\beta$ contains one of the graphs~$\Gamma_{\widetilde T}$,~$\Gamma_{\widetilde E}$ or~$\Gamma_{\widetilde X}$ as an
induced subgraph, then the canonical Seifert surface~$\Sigma(\beta)$ contains~$\widetilde T$,~$\widetilde E$ or~$\widetilde X$, respectively, as a surface minor. 
\end{lem}

\begin{proof}
Suppose the linking pattern~$\mathcal{P}(\beta)$ for a positive braid~$\beta$ contains one of the graphs~$\Gamma_{\widetilde T}$,~$\Gamma_{\widetilde E}$ or~$\Gamma_{\widetilde X}$ as an
induced subgraph. Then, the positive braid~$\beta$ contains a subword~$\beta'$ whose linking pattern is~$\Gamma_{\widetilde T}$,~$\Gamma_{\widetilde E}$ or~$\Gamma_{\widetilde X}$, respectively. 
For~$\beta'$, just take all the letters of~$\beta$ which in the brick diagram define the horizontal edges of the rectangles corresponding to the vertices of the induced subgraph.  
Here, it is important that the subgraph~$\Gamma_{\widetilde T}$,~$\Gamma_{\widetilde E}$ or~$\Gamma_{\widetilde X}$ is induced: if it were not, then the above definition of~$\beta'$ would yield
a braid word whose linking pattern might have more edges.  
 
The canonical Seifert surface~$\Sigma(\beta')$ is a surface minor of~$\Sigma(\beta)$. It only remains to show that the canonical Seifert surface~$\Sigma(\beta')$ equals~$\widetilde T$,~$\widetilde E$ or~$\widetilde X$, respectively. 
This follows from the fact that the linking pattern can be enriched with some additional information, encoded by an edge orientation, so that it uniquely determines the positive braid link type. This is a result due to Baader, Lewark and the author~\cite{BLL}, and this enriched graph is called the \emph{linking graph}.
We note that while the linking graph a priori contains more information than the linking pattern (which is the unoriented version of the linking graph), it does not in case the underlying abstract graph is~$\Gamma_{\widetilde T}$,~$\Gamma_{\widetilde E}$ or~$\Gamma_{\widetilde X}$. 
Indeed, removing any edge of one of the graphs~$\Gamma_{\widetilde T}$,~$\Gamma_{\widetilde E}$ or~$\Gamma_{\widetilde X}$ divides the graph into two connected components, at least one of which is symmetric with respect to a reflection in the plane. 
The uniqueness now follows from Corollary~8 in~\cite{BLL}. In particular, we get that~$\widehat\beta'$ has the same link type as the boundary link~$\partial\widetilde T, \partial\widetilde E$ or~$\partial\widetilde X$, respectively, so the canonical Seifert surface~$\Sigma(\beta')$ is isotopic to~$\widetilde T, \widetilde E$ or~$\widetilde X$, respectively. 
\end{proof}

\subsection{The positive braid index}
Recall that the {positive braid index} of a positive braid link is the minimal number of strands necessary to represent the link as the closure of a positive braid.
The following lemma gives a condition under which a positive braid~$\beta$ cannot be of \emph{minimal positive braid index}, that is, does not realise the positive braid index of its closure~$\widehat\beta$.
It is stated also by the author in~\cite{L2}. As the proof there contains a mistake, we give a new proof. 
Let the \emph{subword of a positive braid~$\beta$ induced by a subset~$S$ of generators~$\sigma_i$} be the word obtained from~$\beta$ by deleting all occurrences of generators that are not in~$S$.  

\begin{lem}
\label{minimalitylemma}
Let~$\beta$ be a prime positive braid on at least three strands. 
If for some~$i$, the linking pattern of the subword of~$\beta$ induced by the generators~$\sigma_i$ and~$\sigma_{i+1}$ is a path, then~$\beta$ is not of minimal positive braid index.
\end{lem}

\begin{proof}
We can assume the subword of~$\beta$ induced by the generators~$\sigma_i$ and~$\sigma_{i+1}$ to be~$\sigma_i^k\sigma_{i+1}\sigma_i\sigma_{i+1}^l$, 
for some positive numbers~$k$ and~$l$.
This can be achieved by cyclic permutation and possibly reversing the order of the word~$\beta$, operations that do not change the positive braid index.
Similarly, we can assume that all occurrences of generators with index smaller than~$i$ take place before the last occurrence of~$\sigma_i$,
and, likewise, all occurrences of generators with index greater than~$i+1$ take place after the first occurrence of~$\sigma_{i+1}$. 
The situation is schematically depicted in Figure~\ref{nichtminimal} on the left.
\begin{figure}[h]
\begin{center}
\def\svgwidth{270pt}
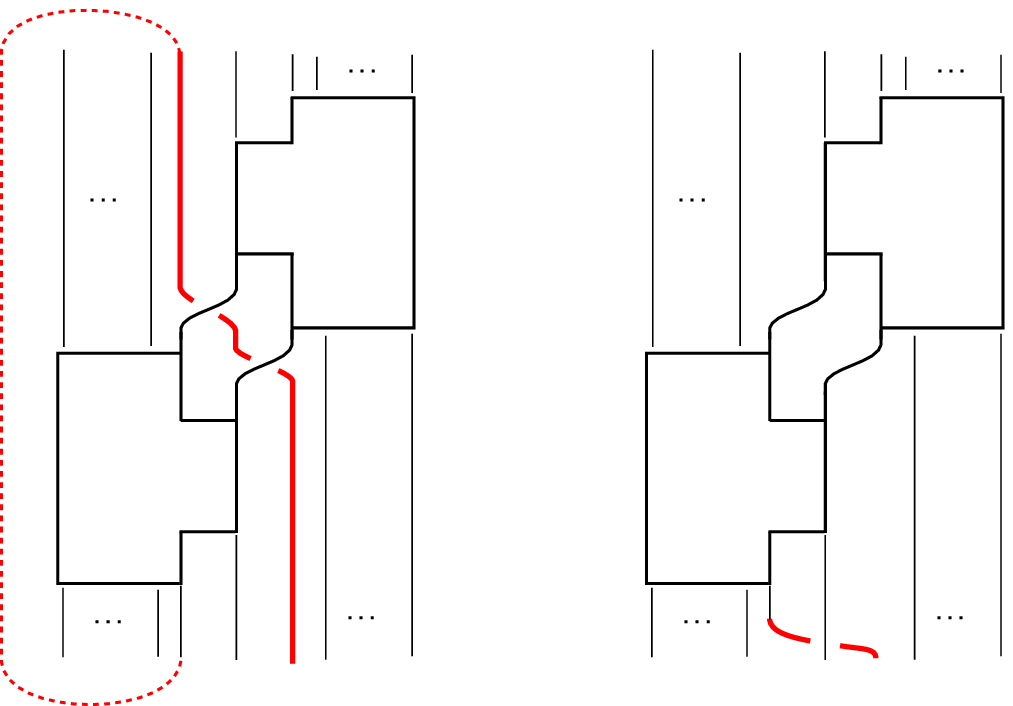
\caption{}
\label{nichtminimal}
\end{center}
\end{figure}
The strand depicted in thick red passes below the two strands it crosses. 
Thus, the closure of~$\beta$ is isotopic to the closure of the braid depicted schematically in Figure~\ref{nichtminimal} on the right.
(In the figure, only the closure of the strand which is moved in the isotopy is shown, in dashed red). 
This braid is still positive but has one strand less than~$\beta$.
\end{proof}

\section{Induced subgraphs of the linking pattern and surface minors}
\label{path_section}

Let~$K$ be a prime knot obtained as the closure of a positive braid~$\beta$ on~$b$ strands. 
We denote by~$\Sigma(\beta)$ and~$\mathcal{P}(\beta)$ the canonical Seifert surface of~$\widehat\beta$ and the linking pattern of~$\beta$, respectively.
For~$I\subset\{1,\dots,b-1\}$, let~$\beta_I$ be the subword of~$\beta$ induced by generators with indices in~$I$.
Furthermore, let~$\Sigma_I(\beta)$ be the surface minor of~$\Sigma(\beta)$ given by the vertical discs and curved handles corresponding to generators with indices in~$I$.
Similarly, we denote by~$\mathcal{P}_I(\beta)$ the subgraph of the linking pattern induced by the vertices corresponding to braid generators with indices in~$I$.
For example, with this notation, Lemma~\ref{minimalitylemma} states that for a positive braid~$\beta$ on~$b\ge3$ strands, 
if~$\mathcal{P}_{\{i,i+1\}}(\beta)$ is a path for some~$1\le i <b-2$, then~$\beta$ is not of minimal positive braid index.
We write~$P_i$ instead of~$P_{\{i\}}$ for index sets with one element.

Our proof method requires us to find the linking patterns corresponding to~$\widetilde T$,~$\widetilde E$ or~$\widetilde X$ 
as induced subgraphs of the linking pattern of prime positive braid knots. 
An important way for us to achieve this is the following. Sometimes, it is possible to find an induced subgraph of the linking pattern with a vertex of degree three. 
In some cases, it is even possible to prolong the arms of this graph (while staying an induced subgraph of the linking pattern) until it is a 
linking pattern corresponding to~$\widetilde T$,~$\widetilde E$ or~$\widetilde X$.
In this context, we make use of the following observation. 

\begin{lem}
\label{inducedpaths_lemma}
Let~$\beta$ be a prime positive braid on~$b$ strands, 
and let~$v$ be a vertex of~$\mathcal{P}_i(\beta)$ for some~$1\le i\le b-1$. 
\begin{enumerate}
\item[(i)] For any natural number~$1\le j \le i-1$,~$\mathcal{P}_{\{i-j,\dots,i\}}$ contains 
an induced path that begins at~$v$ and is of length at least~$j$. 
\item[(ii)] For any natural number~$1\le j\le b-i-1$,~$\mathcal{P}_{\{i,\dots,i+j\}}$ contains
an induced path that begins at~$v$ and is of length at least~$j$. 
\end{enumerate}
\end{lem}

\begin{proof} We give a recipe for finding induced paths starting at a given vertex of the linking pattern (thought of as a brick in the brick diagram)
and going in a chosen direction (right or left) in the standard visualisation of the brick diagram. 
\begin{enumerate}
\item Fix your chosen brick~$v$. Depending on whether the brick~$v$ is linked with a brick in the column on the right (left) or not, proceed with~(3) or~(2), respectively.
\item If the brick~$v$ is not linked with a brick in the column on the right (left), add a brick~$w$ to the path. 
Here,~$w$ is the brick either above or below~$v$, depending on which one is closer to a brick in the same column linking with a brick in the column on the right (left). 
Then go back to~(1) with~$v=w$.
\item If the brick~$v$ is linked with at least one brick on the right (left), add a linked brick~$w$ on the right (left) to the path. Here, the brick~$w$ is chosen to be as close as possible to a brick in
its column that is linked with a brick in the column to its right (left). Then go back to~(1) with~$v=w$.
\end{enumerate}
Choosing the brick closest to a linking brick in step~(3) 
ensures that there is no linking with the former column when adding bricks as in step~(2) until again arriving at step~(3).
\begin{figure}[h]
\begin{center}
\def\svgwidth{110pt}
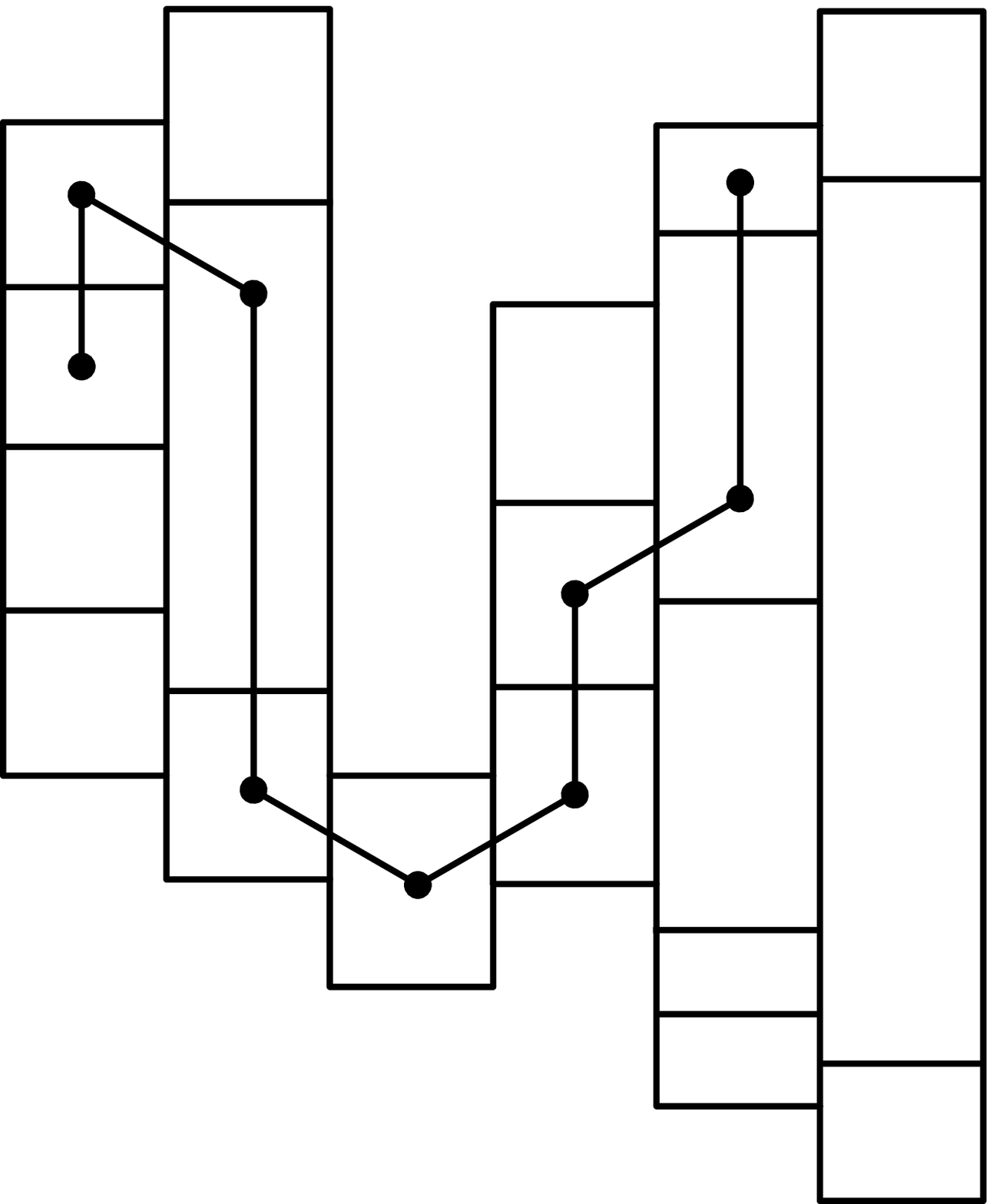
\caption{}
\label{pathexample}
\end{center}
\end{figure}
Figure~\ref{pathexample} illustrates the induced path starting at the endpoint on the left chosen by this recipe for a sample brick diagram. 
\end{proof}

The following lemma shows how we can use Lemma~\ref{inducedpaths_lemma} in order to find surface minors~$\widetilde T$,~$\widetilde E$ or~$\widetilde X$.
We prove several similar statements in Section~\ref{minor_section}.

\begin{lem}
\label{einekante}
Let~$K$ be a prime knot obtained as the closure of a positive braid~$\beta$ of minimal positive braid index~$b$. 
Furthermore, let~$i$ be a natural number such that~$5<i<b-6$. If~$\mathcal{P}_i(\beta)$ and~$\mathcal{P}_{i+1}(\beta)$ are connected by exactly one edge in~$\mathcal{P}(\beta)$,
then~$\Sigma_{\{i-5,\dots,i+6\}}(\beta)$ contains~$\widetilde T$,~$\widetilde E$ or~$\widetilde X$ as a surface minor.
\end{lem}

\begin{proof}
Up to cyclic permutation,~$\beta_{\{i,i+1\}}$ equals~$\sigma_{i}^a\sigma_{i+1}^b\sigma_i^c\sigma_{i+1}^d$ where~$a,b,c,d>0$ are integers (the inequality is strict due to the assumption that the closure is prime).
Either~$a,c\ge2$ or~$b,d\ge~2$. Otherwise,~$\mathcal{P}_{\{i,i+1\}}(\beta)$ is a path and~$\beta$ is not of minimal positive braid index by Lemma~\ref{minimalitylemma}. 
By symmetry, we may assume~$a,c\ge2$.
Assume~$a\ge3$ for a moment. Then,~$\beta$ contains the subword~$\sigma_i^3\sigma_{i+1}\sigma_i^2\sigma_{i+1}$. 
In particular,~$\mathcal{P}_{\{i,i+1\}}(\beta)$ contains the graph~$D_5$ as an induced subgraph, as indicated in Figure~\ref{D5subgraph}. 
\begin{figure}[h]
\begin{center}
\def\svgwidth{30pt}
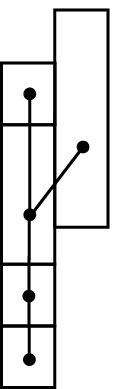
\caption{}
\label{D5subgraph}
\end{center}
\end{figure}
We now apply Lemma~\ref{inducedpaths_lemma}(ii) for~$v$ and~$j=5$. This yields an induced path in~$\mathcal{P}_{\{i,\dots,i+6\}}$ that is of length at least~$5$ and starts at the vertex~$v$. 
Furthermore, no vertex (except for~$v$) of this path is connected to a vertex of~$\mathcal{P}_i(\beta)$ in~$\mathcal{P}(\beta)$. 
This follows from our assumption that~$\mathcal{P}_i(\beta)$ and~$\mathcal{P}_{i+1}(\beta)$ are connected by exactly one edge in~$\mathcal{P}(\beta)$.
As the longest arm of~$\Gamma_{\widetilde E}$ is of length~$6$, 
we obtain that~$\Sigma_{\{i,\dots,i+6\}}(\beta)$ contains a surface minor~$\widetilde E$. 
Therefore, we can now restrict to the case~$a=c=2$. 
By symmetry, if in the beginning we assumed~$b,d\ge2$, using the same strategy we could apply Lemma~\ref{inducedpaths_lemma}(i) 
and find a surface minor~$\widetilde E$ of~$\Sigma_{\{i-5,\dots,i+1\}}(\beta)$.
From now on, we often do not mention this symmetry anymore.

Now, we consider the generator~$\sigma_{i-1}$. Assume that the first occurrence of~$\sigma_{i-1}$ happens before the first occurrence of~$\sigma_i$ in~$\beta$.
There has to be another occurrence of~$\sigma_{i-1}$ before the last occurrence of~$\sigma_i$ in~$\beta$, otherwise~$\mathcal{P}_{\{i-1,i\}}(\beta)$ is disconnected and~$\widehat\beta$ is not prime.
Independently of where this occurrence takes place, we can find a surface minor~$\widetilde E$ or~$\widetilde X$ of~$\Sigma_{\{i-1,\dots,i+6\}}(\beta)$ by adding a path at the vertices~$v$ 
depicted in Figure~\ref{firstoccurrence}, using Lemma~\ref{inducedpaths_lemma}.
\begin{figure}[h]
\begin{center}
\def\svgwidth{200pt}
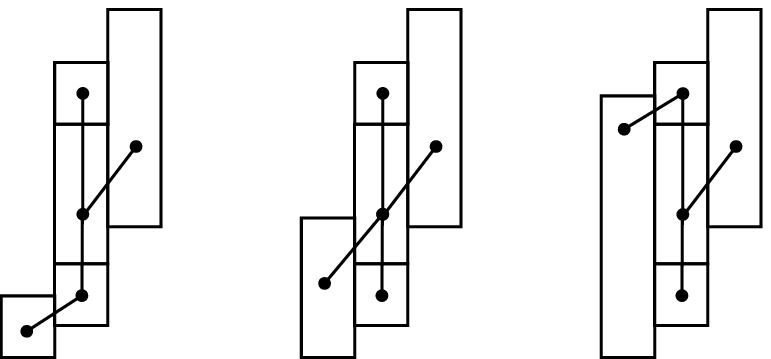
\caption{}
\label{firstoccurrence}
\end{center}
\end{figure}

Now consider Figure~\ref{invariantstrand2}. We have just shown that we may restrict to the case where there are no crossings in the two regions marked with~``X". 
\begin{figure}[h]
\begin{center}
\def\svgwidth{80pt}
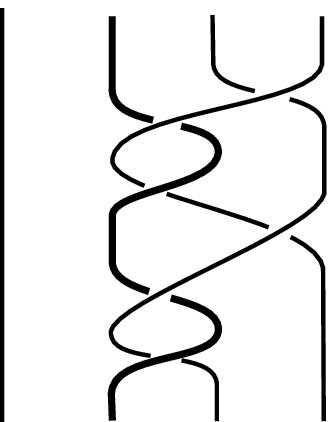
\caption{}
\label{invariantstrand2}
\end{center}
\end{figure}
Hence, there must be at least one crossing in the region marked with~``$\ast$". Otherwise, the~$i$th strand (the thick strand depicted in Figure~\ref{invariantstrand2}) is left invariant by the permutation given by~$\beta$, and~$\widehat\beta$ is not a knot.
Furthermore, in at least one of the two regions marked with~``?", there must be at least one crossing. Otherwise,~$\widehat\beta$ is not prime.
If there is no crossing in the upper of the two regions marked with~``?", we find a surface minor~$\widetilde T$ in~$\Sigma_{\{i-4,\dots,i+4\}}(\beta)$ 
by adding a path at the vertices~$w$ (to the left) and~$v$ (to the right) shown in Figure~\ref{fragezeichenregion}, using Lemma~\ref{inducedpaths_lemma}.
\begin{figure}[h]
\begin{center}
\def\svgwidth{50pt}
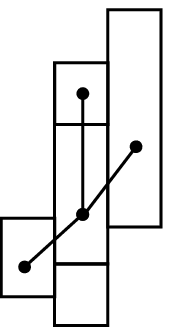
\caption{}
\label{fragezeichenregion}
\end{center}
\end{figure}
Similarly, we find a surface minor~$\widetilde T$ in~$\Sigma_{\{i-4,\dots,i+4\}}(\beta)$ if there is no crossing in the lower of the two regions marked with~``?".

We have shown that we may restrict to the case where, 
up to possibly deleting some generators~$\sigma_{i-1}$,~$\beta_{\{i-1,i,i+1\}}$ 
equals~$\sigma_i\sigma_{i-1}\sigma_i\sigma_{i+1}^b\sigma_{i-1}\sigma_i\sigma_{i-1}\sigma_i\sigma_{i+1}^d$.
Applying two braid relations~$\sigma_i\sigma_{i-1}\sigma_i \to \sigma_{i-1}\sigma_{i}\sigma_{i-1}$ 
yields~$\sigma_{i-1}\sigma_{i}\sigma_{i-1}\sigma_{i+1}^b\sigma_{i-1}^2\sigma_{i}\sigma_{i-1}\sigma_{i+1}^d,$
the linking pattern of which contains~$D_5$ as an induced subgraph, as indicated in Figure~\ref{lastcase}.
Note that by the manipulations we just described, we never change an occurrence of~$\sigma_{i+1}$. 
In particular,~$\Sigma_{\{i-1,\dots,i+6\}}(\beta)$ contains a surface minor~$\widetilde E$ by adding a path starting at the vertex~$v$ in Figure~\ref{lastcase}, using Lemma~\ref{inducedpaths_lemma}.
\begin{figure}[h]
\begin{center}
\def\svgwidth{50pt}
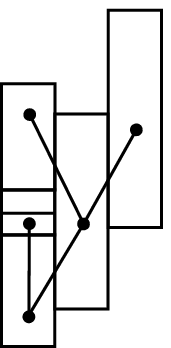
\caption{}
\label{lastcase}
\end{center}
\end{figure}
\end{proof}

The statement of Lemma~\ref{einekante} also holds if~$\mathcal{P}_i(\beta)$ and~$\mathcal{P}_{i+1}(\beta)$ are connected by exactly two edges in~$\mathcal{P}(\beta)$.

\begin{lem}
\label{einekante_lemma}
Let~$K$ be a prime knot obtained as the closure of a positive braid~$\beta$ of minimal positive braid index~$b$. 
Furthermore, let~$i$ be a natural number such that~$5<i<b-6$. If~$\mathcal{P}_i(\beta)$ and~$\mathcal{P}_{i+1}(\beta)$ are connected by exactly two edges in~$\mathcal{P}(\beta)$,
then~$\Sigma_{\{i-5,\dots,i+6\}}(\beta)$ contains~$\widetilde T$,~$\widetilde E$ or~$\widetilde X$ as a surface minor.
\end{lem}

\begin{proof}
If~$\mathcal{P}_i(\beta)$ and~$\mathcal{P}_{i+1}(\beta)$ are connected by exactly two edges in~$\mathcal{P}(\beta)$,
then either~$\beta_{\{i,i+1\}} = \sigma_i^a\sigma_{i+1}^b\sigma_i^c\sigma_{i+1}^d\sigma_i^e$ or~$\beta_{\{i,i+1\}} = \sigma_{i+1}^a\sigma_{i}^b\sigma_{i+1}^c\sigma_{i}^d\sigma_{i+1}^e$.
In particular, up to cyclic permutation of~$\beta$, $\beta_{\{i,i+1\}} = \sigma_i^{a+e}\sigma_{i+1}^b\sigma_i^c\sigma_{i+1}^d$ 
or~$\beta_{\{i,i+1\}} = \sigma_{i+1}^{a+e}\sigma_{i}^b\sigma_{i+1}^c\sigma_{i}^d$, 
respectively. Hence, Lemma~\ref{einekante} gives the existence of a surface minor~$\widetilde T$,~$\widetilde E$ or~$\widetilde X$ of~$\Sigma_{\{i-5,\dots,i+6\}}(\beta)$. 
\end{proof}

\section{Finding minors~$\widetilde T$,~$\widetilde E$ and~$\widetilde X$}
\label{minor_section}

The goal of this section is to give the means for detecting surface minors~$\widetilde T$,~$\widetilde E$ and~$\widetilde X$ of canonical Seifert surfaces of prime positive braid knots.
We establish a series of lemmas in the spirit of Lemma~\ref{einekante} with changing assumptions on the braid~$\beta$, providing such surface minors. 
These lemmas basically constitute a case distinction which allows us to prove Proposition~\ref{keyprop} in Section~\ref{linearityproof_section}, from which we deduce Theorem~\ref{linearity_thm}.  

\begin{lem}
\label{squaresquaresquare}
Let~$K$ be a prime knot obtained as the closure of a positive braid~$\beta$ of minimal positive braid index~$b$.
Furthermore, let~$i$ be a natural number such that~$6<i<b-6$.
If~$\beta_{\{i,i+1\}}$ ends, up to cyclic permutation, with~$\sigma_{i+1}^{c}\sigma_i^{b}\sigma_{i+1}^{a}$ for~$a,b,c\ge2$,
then~$\Sigma_{\{i-6,\dots,i+6\}}(\beta)$ contains~$\widetilde T, \widetilde E$ or~$\widetilde X$ as a surface minor.
\end{lem}

\begin{proof}
We only have to consider the case where there are at least two additional occurrences of~$\sigma_i$ in~$\beta_{\{i,i+1\}}$. 
Otherwise, there cannot be three or more edges between~$\mathcal{P}_i(\beta)$ and~$\mathcal{P}_{i+1}(\beta)$ in~$\mathcal{P}(\beta)$, 
so we are done by Lemma~\ref{einekante} and Lemma~\ref{einekante_lemma}. 
In particular,~$\beta_{\{i,i+1\}}$ contains~$\sigma_i^d\sigma_{i+1}^{c}\sigma_i^{b}\sigma_{i+1}^{a}$ as a subword, where~$a,b,c,d\ge2$. 
If one out of~$a,b,c$ or~$d$ is strictly greater than~$2$, then~$\Sigma_{i,i+1}(\beta)$ contains~$\widetilde X$ as a surface minor, 
compare with the situation in Figure~\ref{thirdcase} (which is obtained by a cyclic permutation and, if need be, taking a subword).
So, we are left with the case where, up to cyclic permutation, we have~$\beta_{\{i,i+1\}} =  \sigma_i\sigma_{i+1}^e\sigma_i\sigma_{i+1}^{2}\sigma_i^{2}\sigma_{i+1}^{2},$ for~$e\ge1$.

\begin{figure}[h]
\begin{center}
\def\svgwidth{35pt}
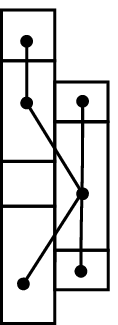
\caption{}
\label{thirdcase}
\end{center}
\end{figure}

Now we consider how the occurrences of the generator~$\sigma_{i-1}$ fit into the fixed subword~$\beta_{\{i,i+1\}} =  \sigma_i\sigma_{i+1}^e\sigma_i\sigma_{i+1}^{2}\sigma_i^{2}\sigma_{i+1}^{2}.$ 
Assume that the first occurrence of~$\sigma_{i-1}$ happens before the first occurrence of~$\sigma_i$ in~$\beta_{\{i-1,i,i+1\}}$.
There has to be another occurrence of~$\sigma_{i-1}$ before the last occurrence of~$\sigma_i$ in~$\beta_{\{i-1,i,i+1\}}$, otherwise~$\mathcal{P}_{\{i-1,i\}}$ is disconnected and~$\widehat\beta$ is not prime.
In each case, we can find a surface minor~$\widetilde X$ of~$\Sigma_{\{i-1,i,i+1\}}$ (by contracting the dotted edge if necessary, compare with Remark~\ref{warumkontrahieren}) as shown in Figure~\ref{firstoccurrence2}.
\begin{figure}[h]
\begin{center}
\def\svgwidth{200pt}
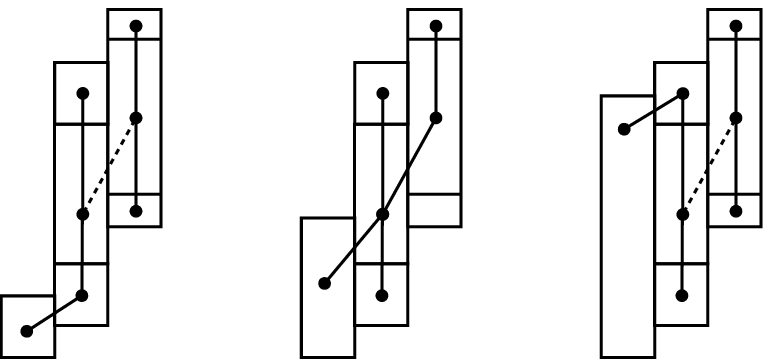
\caption{}
\label{firstoccurrence2}
\end{center}
\end{figure}
By cyclic permutation, an occurrence of~$\sigma_{i-1}$ in~$\beta_{\{i-1,i,i+1\}}$ after the last occurrence of~$\sigma_i$ also yields a surface minor~$\widetilde X$. 
So, let~$\sigma_{i-1}$ occur only after the first occurrence of~$\sigma_i$, but before the last one. Since~$\sigma_i$ occurs (counted with multiplicity) exactly four times in~$\beta_{\{i-1,i,i+1\}}$, 
the only way for~$\mathcal{P}_{i-1}(\beta)$ and~$\mathcal{P}_{i}(\beta)$ to be connected by at least three edges in~$\mathcal{P}(\beta)$ 
is for~$\sigma_{i-1}$ to split every pair of occurrences of~$\sigma_i$ in~$\beta_{\{i-1,i,i+1\}}$.
Otherwise, we are again done by Lemma~\ref{einekante} and Lemma~\ref{einekante_lemma}. 
In this case,~$\beta_{\{i-1,i,i+1\}}$ must contain~$\sigma_i^2\sigma_{i+1}^{2}\sigma_{i-1}\sigma_i\sigma_{i-1}\sigma_i\sigma_{i+1}^{2}$ as a subword. 
Applying the braid relation~$\sigma_i\sigma_{i-1}\sigma_i\to\sigma_{i-1}\sigma_{i}\sigma_{i-1}$ yields the braid word 
$\sigma_i^2\sigma_{i+1}^{2}\sigma_{i-1}^2\sigma_{i}\sigma_{i-1}\sigma_{i+1}^{2}$, whose canonical Seifert surface contains~$\widetilde X$ as a surface minor. 
This can be seen by contracting the dotted edge in Figure~\ref{oh-mann}, compare with Remark~\ref{warumkontrahieren}.
\begin{figure}[h]
\begin{center}
\def\svgwidth{50pt}
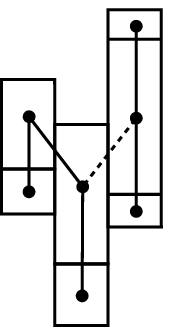
\caption{}
\label{oh-mann}
\end{center}
\end{figure}
\end{proof}

\begin{remark}
\label{warumkontrahieren}
\emph{The linking patterns shown in Figures~\ref{firstoccurrence2} and~\ref{oh-mann} with a dotted edge contain the graph~$\Gamma_{\widetilde X}$ as an induced subgraph, 
but only after contracting the dotted edge. We note that the associated canonical Seifert surfaces also contain the surface~$\widetilde X$ as a surface minor.
\begin{figure}[h]
\begin{center}
\def\svgwidth{110pt}
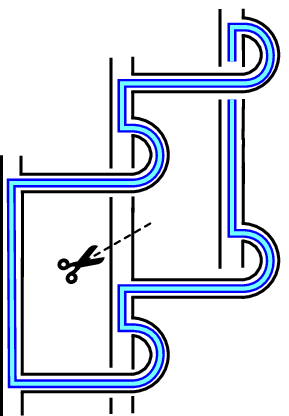
\caption{}
\label{contractionsurface}
\end{center}
\end{figure}
This can be seen by cutting the ribbon of the surface which retracts to the vertical side segment of the brick diagram intersected by the dotted edge, 
as indicated in Figure~\ref{contractionsurface}.
After cutting the ribbon in Figure~\ref{contractionsurface}, instead of the two Hopf bands corresponding to the vertices connected by the dotted edge, 
there is only one Hopf band, which retracts to the Hopf band depicted in blue. 
So, after cutting this vertical ribbon of one of the surfaces corresponding to the intersection patterns with a dotted line in Figure~\ref{firstoccurrence2} or~\ref{oh-mann},
what is left is a plumbing of positive Hopf bands along the tree~$\Gamma_{\widetilde X}$. 
There is only one such surface up to ambient isotopy, namely~$\widetilde X$, see, for example,~\cite{BLL}.
Hence,~$\widetilde X$ is indeed a surface minor. 
}\end{remark}

\begin{lem}
\label{square}
Let~$K$ be a prime knot obtained as the closure of a positive braid~$\beta$ of minimal positive braid index~$b$.
Let~$i$ be a natural number such that~$6<i<b-6$.
Assume furthermore that, up to cyclic permutation,~$\beta_{\{i-1,i,i+1\}}$ ends with~$\sigma_i^2$ and no braid moves 
\begin{align*}
\sigma_i\sigma_{i-1}\sigma_i&\to\sigma_{i-1}\sigma_{i}\sigma_{i-1},\\ 
\sigma_i\sigma_{i+1}\sigma_i&\to\sigma_{i+1}\sigma_{i}\sigma_{i+1}
\end{align*}
can be applied to any cyclic permutation of~$\beta$. 
Then~$\Sigma_{\{i-6,\dots,i+6\}}(\beta)$ contains~$\widetilde T$,~$\widetilde E$ or~$\widetilde X$ as a surface minor.
\end{lem}

\begin{proof}
We first arrange by cyclic permutation that~$\beta_{\{i-1,i,i+1\}}$ does not end with~$\sigma_i^3$ (but still ends with~$\sigma_i^2$).
Now, in case~$\beta_{\{i-1,i,i+1\}}$ ended with~$\sigma_{i-1}\sigma_{i+1}\sigma_i^2$, 
then using Lemma~\ref{inducedpaths_lemma} one could find a surface minor~$\widetilde T$ 
by adding a path at the vertices~$w$ (to the left) and~$v$ (to the right) indicated in Figure~\ref{doublesplit}.
\begin{figure}[h]
\begin{center}
\def\svgwidth{50pt}
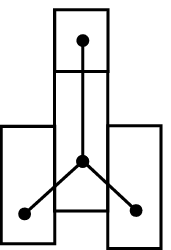
\caption{}
\label{doublesplit}
\end{center}
\end{figure}
So, we can assume without loss of generality (using the symmetry of the situation) that only~$\sigma_{i-1}$ splits the last two occurrences of~$\sigma_i$, that is,
the end~$\sigma_i^3$ of~$\beta_{\{i\}}$ gets split into~$\sigma_i\sigma_{i-1}^a\sigma_i^2$ in~$\beta_{\{i-1,i,i+1\}}$ for some~$a\ge1$.
Actually, this occurrence of~$\sigma_{i-1}$ must be to a power~$a\ge2$, otherwise a braid move~$\sigma_i\sigma_{i-1}\sigma_i\to\sigma_{i-1}\sigma_{i}\sigma_{i-1}$ is possible.
So, we have to consider the case where both~$\beta_{\{i-1,i,i+1\}}$ and~$\beta_{\{i-1,i\}}$ end with~$\sigma_i^b\sigma_{i-1}^a\sigma_i^2$. 
If~$b\ge2$, then we are done by Lemma~\ref{squaresquaresquare}, so we assume~$b=1$ and consider the case where~$\beta_{\{i-1,i\}}$ ends with~$\sigma_{i-1}\sigma_i\sigma_{i-1}^a\sigma_i^2$,
for some~$a\ge2$.
 
We again consider the generator~$\sigma_{i+1}$. If~$\beta_{\{i-1,i,i+1\}}$ ends with~$\sigma_{i+1}\sigma_{i-1}\sigma_i\sigma_{i-1}^a\sigma_i^2$, 
we find a surface minor~$\widetilde E$ by adding a path at the vertex~$v$ shown in Figure~\ref{nocheins}, using Lemma~\ref{inducedpaths_lemma}. 
\begin{figure}[h]
\begin{center}
\def\svgwidth{50pt}
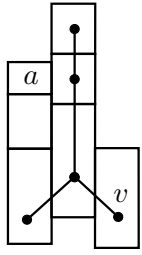
\caption{Here and in other figures later on, a variable~$z$ in a brick without a depicted vertex stands for~$z-2$ additional horizontal segments in the brick.}
\label{nocheins}
\end{center}
\end{figure}
Recall for this that we can assume at least one more occurrence of~$\sigma_{i+1}$ and~$\sigma_i$ earlier in~$\beta$, 
because this is the only way for~$\mathcal{P}_{i+1}(\beta)$ and~$\mathcal{P}_{i}(\beta)$ to be connected by at least three edges 
in~$\mathcal{P}(\beta)$ (otherwise, we are done by Lemma~\ref{einekante} or Lemma~\ref{einekante_lemma}). 
With the same argument, we can also assume another occurrence of~$\sigma_{i-1}$ before the additional occurrence of~$\sigma_i$.
So, we can assume that~$\beta_{\{i-1,i,i+1\}}$ ends with~$\sigma_{i}\sigma_{i-1}^b\sigma_i\sigma_{i-1}^a\sigma_i^2$, where~$a,b\ge2$ , 
since otherwise a braid move~$\sigma_i\sigma_{i-1}\sigma_i\to\sigma_{i-1}\sigma_{i}\sigma_{i-1}$ is possible.

Up to now, we have reduced the proof to the case where the induced subword~$\beta_{\{i-1,i\}}$ ends with 
$\sigma_{i-1}^d\sigma_{i}^c\sigma_{i-1}^b\sigma_i\sigma_{i-1}^a\sigma_i^2$, where~$a,b\ge2$ and~$c,d\ge1$.
Now, we distinguish cases depending on where the last occurrence of~$\sigma_{i+1}$ happens in~$\beta_{\{i-1,i,i+1\}}$.

\emph{Case 1:~$c\ge2$ and the last occurrence of~$\sigma_{i+1}$ splits~$\sigma_i^c$.}
Similarly to what we did in Figure~\ref{nocheins}, we can find a surface minor~$\widetilde E$ by adding a path at the vertex~$v$ shown with Figure~\ref{firstcase}, using Lemma~\ref{inducedpaths_lemma}. 
\begin{figure}[h]
\begin{center}
\def\svgwidth{50pt}
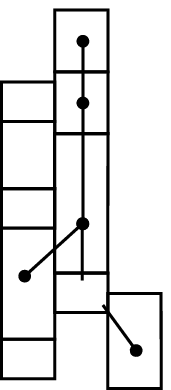
\caption{}
\label{firstcase}
\end{center}
\end{figure}

\emph{Case 2: The last occurrence of~$\sigma_{i+1}$ happens before~$\sigma_i^c$.} 
Again, we can find a surface minor~$\widetilde E$ by adding a path at the vertex~$v$ shown in Figure~\ref{secondcase}, using Lemma~\ref{inducedpaths_lemma}.
\begin{figure}[h]
\begin{center}
\def\svgwidth{50pt}
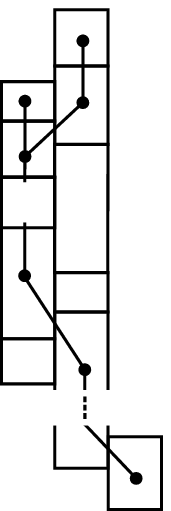
\caption{}
\label{secondcase}
\end{center}
\end{figure}
\end{proof}

\begin{lem}
\label{twooccurrences}
Let~$K$ be a prime knot obtained as the closure of a positive braid~$\beta$ of minimal positive braid index~$b$.
Let~$i$ be a natural number such that~$7<i<b-7$.
Assume furthermore that no braid moves 
\begin{align*}
\sigma_{i-1}\sigma_{i-2}\sigma_{i-1} &\to\sigma_{i-2}\sigma_{i-1}\sigma_{i-2},\\
\sigma_i\sigma_{i-1}\sigma_i &\to\sigma_{i-1}\sigma_{i}\sigma_{i-1},\\ 
\sigma_i\sigma_{i+1}\sigma_i &\to\sigma_{i+1}\sigma_{i}\sigma_{i+1},\\
\sigma_{i+1}\sigma_{i+2}\sigma_{i+1} &\to\sigma_{i+2}\sigma_{i+1}\sigma_{i+2}
\end{align*}
can be applied to any cyclic permutation of~$\beta$. If~$\beta_{\{i,i+1\}}$ has at least two occurrences of~$\sigma_{i+1}$ to a power~$\ge2$,
then~$\Sigma_{\{i-7,\dots,i+7\}}(\beta)$ contains~$\widetilde T$,~$\widetilde E$ or~$\widetilde X$ as a surface minor.
\end{lem}

By symmetry, Lemma~\ref{twooccurrences} also holds if~$\beta_{\{i-1,i\}}$ has at least two occurrences of~$\sigma_{i-1}$ to a power~$\ge2$.

\begin{proof}
Suppose first that at least two occurrences of~$\sigma_{i+1}$ to a power~$\ge2$ in~$\beta_{\{i,i+1\}}$ get split by occurrences of~$\sigma_{i+2}$.
Then, these occurrences of~$\sigma_{i+2}$ must be to a power~$\ge2$, 
otherwise, a braid move~$\sigma_{i+1}\sigma_{i+2}\sigma_{i+1}\to\sigma_{i+2}\sigma_{i+1}\sigma_{i+2}$ is possible.
In particular,~$\beta_{\{i,i+1,i+2\}}$ contains the subword~$\sigma_{i+1}\sigma_{i+2}^2\sigma_{i+1}^2\sigma_{i+2}^2\sigma_{i+1}$.
For~$\mathcal{P}_{i}(\beta)$ and~$\mathcal{P}_{i+1}(\beta)$ to be connected by at least three edges in~$\mathcal{P}(\beta)$, 
there must be at least one more occurrence of~$\sigma_{i+1}$. 
It follows that, up to cyclic permutation,~$\beta_{\{i+1,i+2\}}$ contains the subword~$\sigma_{i+1}\sigma_{i+2}^2\sigma_{i+1}^2\sigma_{1+2}^2\sigma_{i+1}^2$, and hence, 
$\Sigma_{\{i+1,i+2\}}(\beta)$ contains~$\widetilde X$ as a surface minor, compare with Figure~\ref{thirdcase}.

Now, we suppose at most one occurrence of~$\sigma_{i+1}$ to a power~$\ge2$ in~$\beta_{\{i,i+1\}}$ gets split by an occurrence of~$\sigma_{i+2}$.
We think of~$\beta_{\{i,i+1,i+2\}}$ as a product of factors~$(\sigma_{i}^x\sigma_{i+2}^y\sigma_{i+1})$, where~$x,y\ge0$.
There are at least two factors with~$x=0$, since~$\beta_{\{i,i+1\}}$ has at least two occurrences of~$\sigma_{i+1}$ to a power~$\ge2$.
Furthermore, there is at most one factor with~$x=0$ but~$y\ge1$, since we suppose at most one occurrence of~$\sigma_{i+1}$ to a power~$\ge2$ in~$\beta_{\{i,i+1\}}$ gets split by an occurrence of~$\sigma_{i+2}$.
In particular, there is at least one factor~$(\sigma_{i+1})$.
For~$\mathcal{P}_{i+1}(\beta)$ and~$\mathcal{P}_{i+2}(\beta)$ to be connected by at least three edges in~$\mathcal{P}(\beta)$, 
there must be at least three factors with~$y\ge1$. 
Hence, there must be at least two factors with~$x,y\ge1$. 
We now delete every occurrence of~$\sigma_{i+1}$, except the ones from the factor~$(\sigma_{i+1})$, the one from the factor right in front of the factor~$(\sigma_{i+1})$, 
and one of the factors with~$x,y\ge1$.
We do this so that after this deletion, we obtain, up to cyclic permutation,~$\beta_{\{i,i+1,i+2\}} = \sigma_i^d\sigma_{i+2}^c\sigma_{i+1}\sigma_i^b\sigma_{i+2}^a\sigma_{i+1}^2$, for some~$a,b,c,d\ge1$.
Note that even though we have an occurrence of~$\sigma_{i+1}^2$ in~$\beta_{\{i,i+1,i+2\}}$, 
Lemma~\ref{square} does not apply directly, since the condition on the braid moves might not be satisfied anymore, as we deleted some generators~$\sigma_{i+1}$.
However, we can use the argument at the beginning of its proof and find a surface minor~$\widetilde T$ by adding paths at vertices~$w$ (to the left) and~$v$ (to the right) 
that are indicated in Figure~\ref{baldfertig0}.
\begin{figure}[h]
\begin{center}
\def\svgwidth{50pt}
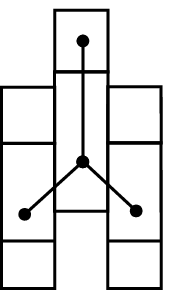
\caption{}
\label{baldfertig0}
\end{center}
\end{figure}
\end{proof}

\section{Linear growth of the genus defect}
\label{linearityproof_section}

We are ready to show that for prime positive braid knots, the genus defect~$g~-~g_4^\mathrm{top}$ grows linearly with the positive braid index. 
The following proposition is all we need to prove Theorem~\ref{linearity_thm}. 

\begin{prop}
\label{keyprop}
Let~$K$ be a prime knot obtained as the closure of a positive braid~$\beta$ of minimal positive braid index~$b$.
Let~$i$ be any natural number such that~$7<i<b-7$.
Then,~$\Sigma_{\{i-7,\dots,i+7\}}(\beta)$ contains~$\widetilde T$,~$\widetilde E$ or~$\widetilde X$ as a surface minor.
\end{prop}

\begin{proof}[Proof of Theorem~\ref{linearity_thm}]

Let~$\Sigma_j = \Sigma_{\{1+16j,\dots,15 + 16j\}}(\beta)$. By Proposition~\ref{keyprop}, 
$\Sigma_j$  contains~$\widetilde T$,~$\widetilde E$ or~$\widetilde X$ as a surface minor for each integer~$j$ such that~$0\le 16j \le b-16$. 
Since the disjoint union of all surfaces~$\Sigma_j$ is an incompressible subsurface of~$\Sigma(\beta)$, we get that~$\Sigma(\beta)$ 
contains a disjoint union of at least~$\left \lfloor{\frac{b}{16}}\right \rfloor$ copies of~$\widetilde T$,~$\widetilde E$ or~$\widetilde X$ as a surface minor. 
Hence,~$g-g_4^\mathrm{top}\ge \left \lfloor{\frac{b}{16}}\right \rfloor$ holds for~$\widehat\beta$,
since genus defect is inherited from surface minors.
\end{proof}

\begin{proof}[Proof of Proposition~\ref{keyprop}]
We start by repeatedly applying a cyclic permutation followed by one of the braid moves below, until there is no cyclic permutation allowing for a possible braid move 
\begin{align*}
\sigma_{i-1}\sigma_{i-2}\sigma_{i-1} &\to\sigma_{i-2}\sigma_{i-1}\sigma_{i-2},\\
\sigma_i\sigma_{i-1}\sigma_i &\to\sigma_{i-1}\sigma_{i}\sigma_{i-1},\\ 
\sigma_i\sigma_{i+1}\sigma_i &\to\sigma_{i+1}\sigma_{i}\sigma_{i+1},\\
\sigma_{i+1}\sigma_{i+2}\sigma_{i+1} &\to\sigma_{i+2}\sigma_{i+1}\sigma_{i+2}
\end{align*}
anymore. 
This process might not be unique. However, it terminates within a finite number of braid moves, 
since each of these braid moves either reduces the sum of powers of generators~$\sigma_i$ or the sum of powers of generators~$\sigma_{i-1},\sigma_i$ and~$\sigma_{i+1}$.
As this process does not change the canonical Seifert surface~$\Sigma_{\{i-7,\dots,i+7\}}(\beta)$, we may assume that~$\beta$ is the result of such a process,
that is, 
no braid move \begin{align*}
\sigma_{i-1}\sigma_{i-2}\sigma_{i-1} &\to\sigma_{i-2}\sigma_{i-1}\sigma_{i-2},\\
\sigma_i\sigma_{i-1}\sigma_i &\to\sigma_{i-1}\sigma_{i}\sigma_{i-1},\\ 
\sigma_i\sigma_{i+1}\sigma_i &\to\sigma_{i+1}\sigma_{i}\sigma_{i+1},\\
\sigma_{i+1}\sigma_{i+2}\sigma_{i+1} &\to\sigma_{i+2}\sigma_{i+1}\sigma_{i+2}
\end{align*}
can be applied to any cyclic permutation of~$\beta$. 

If there is an occurrence of~$\sigma_i$ to a power~$\ge2$ in~$\beta_{\{i-1,i,i+1\}}$, we are done by Lemma~\ref{square}. 
So we assume this is not the case, and think of~$\beta_{\{i-1,i,i+1\}}$ as a product of factors~$(\sigma_{i-1}^x\sigma_{i+1}^y\sigma_i)$,
where either~$x>0$ or~$y>0$.
If there is more than one occurrence of~$\sigma_{i-1}$ or~$\sigma_{i+1}$ to a power~$\ge2$ in~$\beta_{\{i-1,i,i+1\}}$, 
we are done by Lemma~\ref{twooccurrences}. 
So we may assume there is at most one factor~$(\sigma_{i-1}^x\sigma_{i+1}^y\sigma_i)$ with~$x\ge2$ and at most one such factor with~$y\ge2$.
Furthermore, we may assume~$\mathcal{P}_{i}(\beta)$ and~$\mathcal{P}_{i+1}(\beta)$ to be connected by at least three edges in~$\mathcal{P}(\beta)$, 
and likewise for~$\mathcal{P}_{i-1}(\beta)$ and~$\mathcal{P}_{i}(\beta)$, 
since otherwise, we are done by Lemma~\ref{einekante} and Lemma~\ref{einekante_lemma}. 
It follows that~$\beta_{\{i-1,i,i+1\}}$ consists of at least three factors~$(\sigma_{i-1}^x\sigma_{i+1}^y\sigma_i)$.
Note that factors~$(\sigma_{i-1}\sigma_i)$ and~$(\sigma_{i+1}\sigma_i)$ are ruled out by the braid relations we performed at the beginning of the proof.
Indeed, if a factor~$(\sigma_{i-1}\sigma_i)$ or~$(\sigma_{i+1}\sigma_i)$ appeared in~$\beta_{\{i-1,i,i+1\}}$, 
then so would, up to cyclic permutation, a subword~$\sigma_i\sigma_{i-1}\sigma_i$ or~$\sigma_i\sigma_{i+1}\sigma_i$, respectively.
But this would allow for one of the forbidden braid moves. 
In summary, there is at least one factor~$(\sigma_{i-1}\sigma_{i+1}\sigma_i)$ in~$\beta_{\{i-1,i,i+1\}}$.

In case~$\beta_{\{i-1,i,i+1\}}$ ends, up to cyclic permutation, with~$(\sigma_{i-1}^x\sigma_{i+1}^y\sigma_i)(\sigma_{i-1}\sigma_{i+1}\sigma_i)$, 
for numbers~$x,y\ge1$, we find a surface minor~$\widetilde T$ by adding a path at the vertices~$w$ (to the left) and~$v$ (to the right) shown in Figure~\ref{baldfertig}, 
using Lemma~\ref{inducedpaths_lemma}.
\begin{figure}[h]
\begin{center}
\def\svgwidth{50pt}
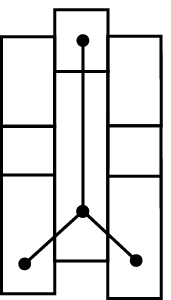
\caption{}
\label{baldfertig}
\end{center}
\end{figure}
Note that in order to obtain the surface minor~$\widetilde T$, we need to add a path to~$w$ and~$v$ not passing through the bricks marked by~$``\ast"$. 
(This is because we are looking for an induced subgraph~$\Gamma_{\widetilde T}$). 
However, if we assume~$\mathcal{P}_{i+1}(\beta)$ and~$\mathcal{P}_{i+2}(\beta)$ to be connected by at least three edges in~$\mathcal{P}(\beta)$, 
this can be achieved for~$v$: the vertex corresponding to the brick marked by~$``\ast"$ has at most two edges connecting to a vertex from~$\mathcal{P}_{i+2}(\beta)$, 
so there must exist at least one other vertex of~$\mathcal{P}_{i+1}$ that is connected by an edge to a vertex of~$\mathcal{P}_{i+2}$.
Similarly, this can be done for~$w$ if~$\mathcal{P}_{i-2}(\beta)$ and~$\mathcal{P}_{i-1}(\beta)$ are connected by at least three edges in~$\mathcal{P}(\beta)$. 
If either~$\mathcal{P}_{i+1}(\beta)$ and~$\mathcal{P}_{i+2}(\beta)$ or~$\mathcal{P}_{i-2}(\beta)$ and~$\mathcal{P}_{i-1}(\beta)$ 
are not connected by at least three edges in~$\mathcal{P}(\beta)$, we are already done by Lemma~\ref{einekante} and Lemma~\ref{einekante_lemma}. 

So far, we have shown that we can assume the factors before and after (in the cyclic order) the factor~$(\sigma_{i-1}\sigma_{i+1}\sigma_i)$ to be~$(\sigma_{i-1}^a\sigma_i)$ and~$(\sigma_{i+1}^b\sigma_i)$, 
respectively, for~$a,b\ge2$.
In this case, we may assume there is at least one other factor~$(\sigma_{i-1}^x\sigma_{i+1}^y\sigma_i)$. 
Otherwise,~$\mathcal{P}_{i}(\beta)$ and~$\mathcal{P}_{i+1}(\beta)$ are connected by only two edges in~$\mathcal{P}(\beta)$, and we are done by Lemma~\ref{einekante_lemma}.
For this factor, only~$x=y=1$ is possible, since we already assumed there is at most one occurrence of~$\sigma_{i-1}$ or~$\sigma_{i+1}$ to a power~$\ge2$ in~$\beta_{\{i-1,i,i+1\}}$. 
In case there was more than one such additional factor, 
~$\beta_{\{i-1,i,i+1\}}$ would contain, up to cyclic permutation, subsequent factors~$(\sigma_{i-1}^x\sigma_{i+1}^y\sigma_i)(\sigma_{i-1}\sigma_{i+1}\sigma_i)$, 
for~$x,y\ge1$, a case we have already dealt with.
Altogether, we may assume that, up to cyclic permutation,~$\beta_{\{i-1,i,i+1\}}$ 
is given by~$(\sigma_{i+1}^b\sigma_i)(\sigma_{i-1}\sigma_{i+1}\sigma_i)(\sigma_{i-1}^a\sigma_i)(\sigma_{i-1}\sigma_{i+1}\sigma_i)$,
for~$a,b\ge2$. The corresponding brick diagram is depicted in Figure~\ref{zweitletztes}.
\begin{figure}[h]
\begin{center}
\def\svgwidth{50pt}
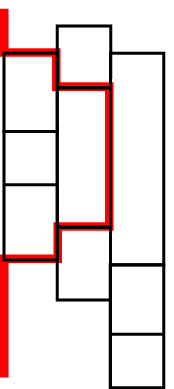
\caption{}
\label{zweitletztes}
\end{center}
\end{figure}

We now consider the generator~$\sigma_{i-2}$. 
We first note that in~$\beta_{\{i-2,i-1,i,i+1\}}$, there must be an occurrence of~$\sigma_{i-2}$ either before the first occurrence of~$\sigma_{i-1}$ or after the last occurrence of~$\sigma_{i-1}$.
Otherwise, the~$i-1$st strand is left invariant by the permutation given by~$\beta$ (as depicted in thick red in Figure~\ref{zweitletztes}), and~$\widehat\beta$ is not a knot. 
Up to cyclic permutation, we can assume there is an occurrence of~$\sigma_{i-2}$ after the last occurrence of~$\sigma_{i-1}$.

There must be other occurrences of~$\sigma_{i-2}$ splitting occurrences of~$\sigma_{i-1}$. 
Otherwise,~$\mathcal{P}(\beta)$ is disconnected and~$\widehat\beta$ is not prime.
We now distinguish cases depending on where occurrences of~$\sigma_{i-2}$ happen.

\emph{Case 1: the occurrence of~$\sigma_{i-1}^a$ in~$\beta_{\{i-1,i,i+1\}}$ is split by~$\sigma_{i-2}$.}
The occurrence of~$\sigma_{i-2}$ is to a power~$\ge2$, otherwise a braid move~$\sigma_{i-1}\sigma_{i-2}\sigma_{i-1}\to\sigma_{i-2}\sigma_{i-1}\sigma_{i-2}$ is possible.
In particular,~$\beta_{\{i-2,i-1\}}$ contains~$\sigma_{i-1}^2\sigma_{i-2}^2\sigma_{i-1}^2\sigma_{i-2}$ as a subword and we can find a surface minor~$\widetilde E$ of~$\Sigma_{\{i-2,\dots,i+4\}}(\beta)$
by adding a path at the vertex~$v$ indicated in Figure~\ref{FERTIG}, using Lemma~\ref{inducedpaths_lemma}. 
\begin{figure}[h]
\begin{center}
\def\svgwidth{60pt}
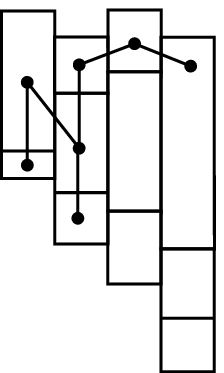
\caption{}
\label{FERTIG}
\end{center}
\end{figure}

\emph{Case 2: the occurrence of~$\sigma_{i-1}^a$ in~$\beta_{\{i-1,i,i+1\}}$ is not split by~$\sigma_{i-2}$.}
In this case, for~$\mathcal{P}_{i-2}(\beta)$ and~$\mathcal{P}_{i-1}(\beta)$ to be connected by at least three edges in~$\mathcal{P}(\beta)$,~$\beta_{\{i-2,i-1\}}$ contains~$\sigma_{i-1}\sigma_{i-2}\sigma_{i-1}^2\sigma_{i-2}\sigma_{i-1}\sigma_{i-2}$ as a subword. Thus, we can find a surface minor~$\widetilde E$ of~$\Sigma_{\{i-2,\dots,i+5\}}(\beta)$
by adding a path at the vertex~$v$ indicated in Figure~\ref{DochNochEins}, using Lemma~\ref{inducedpaths_lemma}.
\begin{figure}[h]
\begin{center}
\def\svgwidth{60pt}
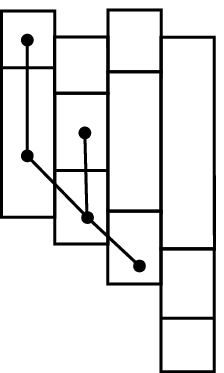
\caption{}
\label{DochNochEins}
\end{center}
\end{figure}
\end{proof}

\begin{remark}
\label{nolowerbound}
\emph{
Theorem~\ref{linearity_thm} does not hold for prime positive braid links. 
In fact, for prime positive braid links, there exists no nontrivial lower bound for the genus defect~$g-g_4^\mathrm{top}$ in terms of the minimal positive braid index.
As examples, we consider the positive braids  $$\beta_k=(\sigma_1\dots\sigma_{k}\sigma_{k}\dots\sigma_1)^2$$
on~$k+1$ strands. The positive braid link~$\widehat\beta_k$ is visually prime and hence prime by a theorem of Cromwell~\cite{Cro}. Furthermore,~$\widehat\beta_k$ is a link with~$k+1$ components. Therefore, the positive braid link~$\widehat\beta_k$ needs at least~$k+1$ strings to be represented as a braid, and hence is of positive braid index~$k+1$.\medskip\newline
\textit{Claim:}~$\vert\sigma(\widehat\beta_k)\vert=2k+1$ and~$\mathrm{null}(\widehat\beta_k)=k-1$. \medskip\newline
Assuming the claim for just a moment, we have~$b_1(\widehat\beta_k) = \vert\sigma(\widehat\beta_k)\vert + \mathrm{null}(\widehat\beta_k)$, so~$\widehat\beta_k$ is of maximal topological 4-genus, that is,~$g=g_4^\mathrm{top}$, by the lower bound of Kauffman and Taylor~\cite{KT}. 
\medskip\newline
In order to prove the claim, we use the canonical Seifert surface for positive braid links.
Let~$S_k$ be the symmetrised Seifert form for~$\Sigma(\beta_k)$. By definition,~$\sigma(\widehat\beta_k)$ is the signature of~$S_k$, and 
~$\mathrm{null}(\widehat\beta_k)$ is the nullity of~$S_k$.
For the calculation, let~$e_1, \dots, e_{3k}$ be the basis of the first homology induced by the bricks of the brick diagram, 
where all curves are oriented anticlockwise and numbered as indicated in Figure~\ref{linkex} for~$k=4$. 
\begin{figure}[h]
\begin{center}
\def\svgwidth{90pt}
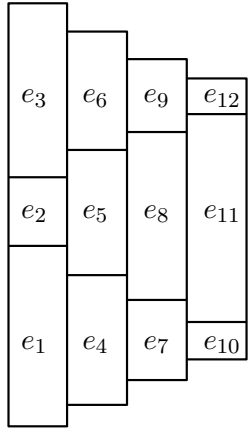
\caption{The brick diagram for the braid~$\beta_4$.}
\label{linkex}
\end{center}
\end{figure}
Inspecting the canonical Seifert surface yields the following description of~$S_k$:
\begin{align*}
S_k(e_i,e_i)&= -2 \ \mathrm{for\ all}\ i,\\
S_k(e_{i\pm1}, e_i)=S_k(e_i,e_{i\pm1}) &= 1,\ \mathrm{for}\ i\equiv2\ (\mathrm{mod}\ 3),\\
S_k(e_{i-2},e_i)=S_k(e_i,e_{i-2}) &= -1,\ \mathrm{for}\ 2\ne i\equiv2\ (\mathrm{mod}\ 3),\\
S_k(e_{i-4}, e_i)=S_k(e_i,e_{i-4}) &= 1, \ \mathrm{for}\ 2\ne i\equiv2\ (\mathrm{mod}\ 3),\\
S_k(e_i,e_j) &=0,\ \mathrm{otherwise}. 
\end{align*}
Let~$N_k$ be the subspace of the first homology of~$\Sigma(\widehat\beta_k)$ generated by~$e_2$ and all~$e_i$ so that~$i\not\equiv 2\ (\mathrm{mod}\ 3)$.
We see that~$S_k$ restricted to~$N_k$ is negative definite. 
Furthermore, let~$O_k$ be the subspace of the first homology of~$\Sigma(\widehat\beta_k)$ generated by all~$x_i$, where we define~$x_i = e_i + \frac{1}{2}(e_{i-1}+e_{i+1}+e_{i-4}-e_{i-2})$, for~$2\ne i\equiv 2\ (\mathrm{mod}\ 3).$ One can compute~$S_k(x_i,x_j)=0$ for all~$2\ne i\equiv 2\ (\mathrm{mod}\ 3)$ and~$2\ne j\equiv 2\ (\mathrm{mod}\ 3)$, so~$S_k$ restricted to~$O_k$ is trivial. 
Furthermore,~$S_k(x_i, e_j)=0$ for all natural numbers~$i$ such that~$2\ne i\equiv 2\ (\mathrm{mod}\ 3)$ and~$j$ either
equal to~$2$ or~$\not\equiv 2\ (\mathrm{mod}\ 3).$ We obtain that the nullity of~$S_k$ equals the dimension of~$O_k$, which is~$k-1$. Furthermore, the absolute value of the signature of~$S_k$ equals the dimension of~$N_k$, which is~$2k+1$. This proves the claim.
}\end{remark}

\section{Surface minor theory for the genus defect}
\label{minortheory_section}
In this section, we deduce the surface minor theoretic applications of Theorem~\ref{linearity_thm}.
More precisely, we show that among prime positive braid knots, having at most a certain genus defect~$g-g_4^\mathrm{top}$ can be characterised by finitely many forbidden surface minors 
of the canonical Seifert surface.

\begin{lem}
\label{16c+32}
Let~$K$ be a prime knot obtained as the closure of a positive braid~$\beta$ of minimal positive braid index~$b$. 
Then~$g-g_4^\mathrm{top}\le c$ holds for~$K=\widehat\beta$ if and only if it holds for~$\widehat\beta_{\{1,\dots, 16c+15\}}$,
where we regard~$\beta_{\{1,\dots, 16c+15\}}$ as a braid on~$\mathrm{min}(b,16c+16)$ strands.
\end{lem}

\begin{proof}
If~$b\le 16c+16$, then~$\beta = \beta_{\{1,\dots, 16c+15\}}$ and the statement is obviously true. 
Now let~$b>16c+16$. By Proposition~\ref{keyprop}, both~$\Sigma(\beta)$ and~$\Sigma(\beta_{\{1,\dots, 16c+15\}})$ 
contain a disjoint union of at least~$c+1$ copies of~$\widetilde T$,~$\widetilde E$ or~$\widetilde X$ as a surface minor.
Hence,~$g-g_4^\mathrm{top} > c$ holds for both~$K$ and~$\widehat\beta_{\{1,\dots, 16c+15\}}$.
\end{proof}

\begin{proof}[Proof of Theorem~\ref{forbiddenminorchar}]
We show that among positive braids of minimal positive braid index whose closure is a prime knot,~$g-g_4^\mathrm{top} \le c$ can be characterised by finitely many forbidden subwords for any~$c\ge0$.
This implies the result on the level of surface minors of canonical Seifert surfaces, since every prime positive braid knot can be written as the closure of a positive braid of minimal positive braid index,
while the associated canonical Seifert surface does not change its isotopy type.
Furthermore, the forbidden surface minors are simply given by the canonical Seifert surfaces (described in Section~\ref{background}) associated with the forbidden subwords. 
For this to make sense, recall that if~$\beta'$ is a subword of a positive braid~$\beta$, then~$\Sigma(\beta')$ is a surface minor of~$\Sigma(\beta)$.

Consider the collection~$P_n$ of positive braid words on~$n+1$ strands whose closures are prime links. 
By Higman's Lemma, the words in a finite alphabet are well-quasi-ordered by the subword partial order~\cite{Higman}. 
In particular, also~$P_n$ is well-quasi-ordered by the subword partial order. 
Since subwords induce surface minors, genus defect~$g-g_4^\mathrm{top} > c$ of the closure is inherited from subwords in~$P_n$.
Equivalently, the property that the positive braid closure has~$g-g_4^\mathrm{top} \le c$ is passed on to subwords in~$P_n$. 
In particular, the property to have genus defect~$g-g_4^\mathrm{top} \le c$ is characterised by finitely many forbidden subwords for~$P_n$. 
Here, we use that a property that is passed on to minors with respect to some well-quasi-order is characterised by finitely many forbidden minors. 

Now, let~$K$ be a prime knot obtained as the closure of a positive braid~$\beta$ of minimal positive braid index~$b$, where~$b$ can be arbitrarily large.
We argue that if~$g-g_4^\mathrm{top} > c$ holds for~$K$, then~$\beta$ must contain one of the finitely many forbidden subwords characterising~$g-g_4^\mathrm{top} \le c$ for~$P_{\mathrm{min}(b-1,16c+15)}$. 
To see this, note that a genus defect~$g-g_4^\mathrm{top} > c$ for~$K$ implies~$g-g_4^\mathrm{top} > c$ for~$\widehat\beta_{\{1,\dots, 16c+15\}}$ by Lemma~\ref{16c+32}. 
In particular, we have that~$\beta_{\{1,\dots, 16c+15\}}$ contains one of the forbidden subwords characterising genus defect~$g-g_4^\mathrm{top} \le c$ for~$P_{\mathrm{min}(b-1,16c+15)}$,
and hence so does~$\beta$.
It follows that the finitely many forbidden subwords characterising the property~$g-g_4^\mathrm{top} \le c$ for braids in~$P_1,P_2,\dots,P_{16c+15}$ suffice to characterise the property~$g-g_4^\mathrm{top} \le c$ among all positive braids
of minimal positive braid index 
whose closure is a prime knot. 
\end{proof}

\end{document}